\documentclass[11pt,psamsfonts]{amsart}
\usepackage{amsmath}
\usepackage{amsthm}
\usepackage{amssymb}
\usepackage{amscd}
\usepackage{amsfonts}
\usepackage{amsbsy}
\usepackage[latin1]{inputenc}
\usepackage{epsfig,afterpage}
\usepackage[dvips]{psfrag}
\usepackage{psfrag}
\usepackage[all]{xy}
\usepackage{color}
\usepackage{subfigure}
\usepackage[update,prepend]{epstopdf}

\newcommand{\vs}{\vspace{0,5cm}}
\usepackage{overpic}

\newcommand{\R}{\ensuremath{\mathbb{R}}}

\newtheorem {proposition}  {Proposition}

\newtheorem {remark} {Remark}

\begin{document}



\title[Hopf and Homoclinic Bifurcations in the Boost Power Converter ]{Hopf and Homoclinic Loop Bifurcations on a DC-DC Boost Converter under a SMC Strategy}

\author[R. Cristiano, T. Carvalho, D.J. Tonon and D.J. Pagano]
{Rony Cristiano$^1$, Tiago Carvalho$^2$, Durval José Tonon$^3$ and\\ Daniel Juan Pagano$^1$}

\address{$^1$ Department of Automation and Systems,
Federal University of Santa Catarina, Florianópolis, Brazil.}

\address{$^2$ FC--UNESP, CEP 17033--360, Bauru, São Paulo, Brazil}

\address{$^3$ Institute of Mathematics and Statistics of Federal University of Goiás, CEP 74001-970, Goi\^{a}nia, Goi\'{a}s, Brazil.}

\email{rony.cristiano@ufsc.br}

\email{tcarvalho@fc.unesp.br}

\email{djtonon@ufg.br}

\email{daniel.pagano@ufsc.br}


\date{}
\dedicatory{} \maketitle


%
%
%
%
%
%
%
%

\begin{abstract}
In this paper, a dc-dc boost converter with sliding mode control and washout filter is analysed. This device is modelled as a three-dimensional Filippov system, characterized by the existence of sliding movement and restricted to the switching manifold. The operating point of the boost converter is a pseudo-equilibrium, and it, undergoes a subcritical Hopf bifurcation. Such a bifurcation occurs in the sliding vector field and creates, in this field, an unstable limit cycle. The limit cycle is confined to the switching manifold and disappears when it touches the visible-invisible two-fold point, resulting in a homoclinic loop which itself closes in this two-fold point.
\end{abstract}

\section{Introduction}\label{secao introducao}

Power electronics devices switched are strongly non-linear and can be modelled as piecewise smooth dynamical systems. It has been shown that this class of systems can exhibit various types of complex phenomena, including the classic bifurcations (Hopf, Saddle-Node, Homoclinic, etc.) and the typical bifurcations induced by discontinuity \cite{diBernardo-livro}.
 
In the case where the dynamical system is discontinuous piecewise smooth, orbits can be confined to the switching manifold. This phenomenon is known as sliding motion and this system class called Filippov systems \cite{Fi}. The occurrence of such a phenomenon has been reported by various applications involving sliding mode control. Here we highlight the applications in power electronics converters \cite{Ponce-Pagano2, Castilla, Siew-Chong}.

In this paper, we study the Hopf and Homoclinic loop bifurcations which occur in the sliding vector field in three-dimensional Filippov systems. For this study we consider the model of a dc-dc boost power electronics converter with sliding mode control and filter washout (SMC-Washout). This bifurcations on sliding vector field are analogous to the standard classic case. However, Homoclinic loop bifurcation differs somewhat standard, because the closing point of homoclinic loop is not on saddle equilibrium point, but in a visible-invisible two-fold singularity that has dynamics saddle in sliding region.

Dynamical systems that have a two-fold singularity possess a very rich and complex dynamics. In \cite{Jeffrey-colombo, Jeffrey-colombo-2011, J-T-T1, J-T-T2} two-fold singularities are studied and  in \cite{diBernardo-electrical-systems, Jeffrey-T-sing} applications of such theory in electrical and control systems, respectively, are exhibited.

The Hopf bifurcation is a local bifurcation in which an equilibrium point of a smooth dynamical system loses stability when a pair of complex conjugate eigenvalues crosses the imaginary axis of the complex plane. In this case, an unstable limit cycle (subcritical Hopf) or stable (supercritical Hopf) bifurcates of equilibrium point.

The Homoclinic loop bifurcation is a global bifurcation that occurs when a limit cycle collides with a saddle equilibrium point. The existence of a homoclinic orbit implies global changes in system dynamics. On bidimensional systems studied by Andronov et al. \cite{A-II}, the existence of a homoclinic orbit causes the sudden appearance of a limit cycle with same stability of the homoclinic orbit. In another sense, we can say that Homoclinic loop bifurcation is the means by which a limit cycle, created in the Hopf bifurcation for example, is destroyed. For details about the Hopf and Homoclinic loop bifurcation in piecewise smooth dynamical systems, see \cite{Kuznetsov-livro, Meiss-livro, Perko-livro}.

In the literature there are many works on Homoclinic loop bifurcation and Hopf bifurcation in non-smooth dynamical systems. Here we highlight the work \cite{Llibre, BinXu, Liping, Simpson2, Dercole, Kuznetsov, Battelli, Eu-fold-sela}. In both, the bifurcations are induced by discontinuity. For example, in Kuznetsov et al. \cite{Kuznetsov}, is studying the ``Pseudo-Homoclinic" bifurcations where a standard saddle may have a homoclinic loop containing a sliding segment. In Dercole et al. \cite{Dercole}, we emphasize the study of ``Boundary Hopf", a bifurcation of codimension 2 where the limit cycle created is linked with a Boundary Equilibrium Bifurcation (BEB) \cite{Pagano}.

The main result of this paper is the report of a new form for the Homoclinic loop bifurcation, where the closing point of homoclinic loop is a two-fold singularity. In the pioneering work, Benadero et al. \cite{Benadero} showed that this phenomenon is present in the system of interconnected power converters in an islanded direct current (DC) microgrid. What are we going do now, is a more rigorous analysis of Hopf and Homoclinic Loop bifurcations from a study case in a power electronics system. 

Following this paper, we discuss basic concepts of Filippov theory in Section \ref{secao preliminares}; we present the model of the boost converter with SMC-Washout, we analyze the tangential singularities and the dynamic of the sliding vector field in Section \ref{section analise conversor}; and finally, we demonstrate the occurrence of Hopf and Homoclinic loop bifurcations in Section \ref{secao bifurcacoes}. The bifurcations analysis is of great importance from the standpoint of global stability and of  the robust control for this type of power converter.


\section{Previous Result}\label{secao preliminares}

\subsection{Filippov's Convention}\label{Filippov}

Let $A \subset \R^3$ be an open set and $$\Sigma = \{(x,y,z)\in A \, | \, h(x,y,z)=0 \},$$ with $h(x,y,z)=z$. Clearly the \textit{switching manifold} $\Sigma$ is the separating boundary of the regions $\Sigma_+=\{(x,y,z) \in A \, | \, z > 0\}$ and $\Sigma_-=\{(x,y,z) \in A \, | \, z < 0\}$.

We define $\mathfrak{X}^r$ the space of $C^r$-vector fields on $A$ endowed with the $C^r$-topology with $r=\infty$ or $r\geq 1$ large enough for our purposes. Call \textbf{$\Omega^r$} the space of vector fields $\mathbf{f}: A \rightarrow \R ^{3}$ such that
\[
\mathbf{f}(\mathbf{x})=\left\{\begin{array}{l} \mathbf{f}^{+}(\mathbf{x}),\quad $for$ \quad \mathbf{x} \in \Sigma_+,\\ 
 \mathbf{f}^-(\mathbf{x}),\quad $for$ \quad \mathbf{x} \in \Sigma_-,
\end{array}\right.
\]
where $\mathbf{x}=(x,y,z)\in A, \mathbf{f}^{\pm}=(f_1^{\pm},f_2^{\pm},f_3^{\pm})\in \mathfrak{X}^r.$ We may consider $\Omega^r = \mathfrak{X}^r \times \mathfrak{X}^r$ endowed with the product topology  and denote any element in $\Omega^r$ by $\mathbf{f}=(\mathbf{f}^{+},\mathbf{f}^-)$, which we will accept to be multivalued in points of $\Sigma$. 

The kind of contact of smooth vector fields $\mathbf{f}^{\pm}\in \mathfrak{X}^r$ with $\Sigma$ are provided by the directional Lie derivatives: 
$$
L_{\mathbf{f}^{\pm}}h = \langle \nabla h, \mathbf{f}^{\pm}\rangle =f_{3}^{\pm} ,
$$
where $\nabla h$ and $\langle.,.\rangle$ denote the gradient of smooth function $h$ and the canonical inner product, respectively. The higher order Lie derivatives are given by $L^m_{\mathbf{f}^{\pm}}h =  \langle \nabla L_{\mathbf{f}^{\pm}}^{m-1} h, \mathbf{f}^{\pm}\rangle$.


On $\Sigma$ we distinguish the following regions: \begin{itemize}
\item Crossing regions, defined by 
$\Sigma^{c+}= \{ \mathbf{x} \in \Sigma \, | \,
f_3^{+}(\mathbf{x})>0, f_3^-(\mathbf{x})>0 \}$ and $\Sigma^{c-} = \{ \mathbf{x} \in \Sigma \, | \,
f_3^{+}(\mathbf{x})<0, f_3^-(\mathbf{x})<0\}$;
\item Sliding region, defined by
$\Sigma^{s}= \{ \mathbf{x} \in \Sigma \, | \, f_3^{+}(\mathbf{x})<0, f_3^-(\mathbf{x})>0 \}$;
\item Escaping region, defined by
$\Sigma^{e}= \{ \mathbf{x} \in \Sigma \, | \, f_3^{+}(\mathbf{x})>0, f_3^-(\mathbf{x})<0\}$.
\end{itemize}  

When $\mathbf{x} \in \Sigma^s$, following the Filippov's convention (see \cite{Fi}), the \textit{sliding vector field} associated to $\mathbf{f}\in \Omega^r$ is the vector field  $\widehat{\mathbf{f}}^s$ tangent to
$\Sigma$ expressed in coordinates as
\begin{equation}
\widehat{\mathbf{f}}^s(\mathbf{x})= \dfrac{1}{(f_3^- - f_3^+)(\mathbf{x})} \begin{bmatrix}
(f_1^+ f_3^- - f_1^- f_3^+)(\mathbf{x})\\
(f_2^{+} f_3^- - f_2^- f_3^+)(\mathbf{x})\\
0
\end{bmatrix}.\label{eq campo filippov}
\end{equation}
Associated to \eqref{eq campo filippov} there exists the \textit{planar normalized sliding vector field}
\begin{equation}
\mathbf{f}^{s}(x,y)=\begin{bmatrix}
(f_1^+ f_3^- - f_1^- f_3^+)(x,y)\\
(f_2^+ f_3^- - f_2^- f_3^+)(x,y)
\end{bmatrix}.\label{equacao campo normalizado}
\end{equation}
Note that, if $\mathbf{x}\in \Sigma^s$ then $f_3^+(x,y)<0$ and $f_3^-(x,y)>0$. So, $(f_3^- - f_3^+)(x,y)>0$ and therefore, $\widehat{\mathbf{f}}^{s}$ and $\mathbf{f}^{s}$ are topologically equivalent in $\Sigma^s$, $\mathbf{f}^s$ has the same orientation as $\widehat{\mathbf{f}}^{s}$ and it can be C$^r$-extended to the closure $\overline{\Sigma^s}$ of $\Sigma^s$.

The points $\mathbf{q} \in \Sigma$ such that $\widehat{\mathbf{f}}^s(\mathbf{q})=\mathbf{0}$ are called \textit{pseudo-equilibria} of $\mathbf{f}$, virtual if $\mathbf{q} \in \Sigma^c$ or real if $\mathbf{q} \in \Sigma^s\cup\Sigma^e$. The points $\mathbf{p} \in \Sigma$ such that $f_3^+(\mathbf{p})\cdot f_3^-(\mathbf{p}) =0$ are called \textit{tangential singularities} of $\mathbf{f}$ (i.e., the trajectory through $\mathbf{p}$ is tangent to $\Sigma$). Furthermore, a point $\mathbf{p} \in \Sigma$ is called \textit{double tangency point} (i.e., the trajectories of both vector fields $\mathbf{f}^{\pm}$ through $\mathbf{p}$ are tangent to $\Sigma$) if $f_3^+(\mathbf{p})= f_3^-(\mathbf{p}) =0$.

 \begin{remark}
	If $\mathbf{q}=(x_q,y_q,0)\in\Sigma^s$ is a pseudo-equilibrium point and $\mathbf{p}=(x_p,y_p,0)$ is a double tangency point, then $\mathbf{f}^{s}(x_q,y_q)=\mathbf{0}$ and $\mathbf{f}^{s}(x_p,y_p)=\mathbf{0}$, i.e., the projections of pseudo-equilibrium and double tangency points on $\Sigma$ are equilibria points of planar normalised sliding vector field $\mathbf{f}^{s}(x,y)$. Furthermore, if $(x_q,y_q)$ is an equilibrium node, focus or saddle, then the pseudo-equilibrium $\mathbf{q}$ is said to be a pseudo-node, pseudo-focus or pseudo-saddle, respectively.
	
\end{remark} 

\subsection{Tangential Singularities}\label{secao-singularidades}

In our approach we deal with two important distinguished tangential singularities: the points where the contact between the trajectory of $\mathbf{f}^+$ or $\mathbf{f}^-$ with $\Sigma$ is either quadratic or cubic, which are called \textit{fold} and \textit{cusp} singularities, respectively (see \cite{Tiago-Marco-FCnaoEstavel}). Observe that this  contact is characterized by Lie's derivative, defined in the previous section. 

A point $\mathbf{p}\in \Sigma$ is a fold point of $\mathbf{f}^+$ if $L_{\mathbf{f}^{+}}h(\mathbf{p})=0$ and $L^2_{\mathbf{f}^{+}}h(\mathbf{p})\neq 0$. Moreover, $\mathbf{p}$ is a cusp point of $\mathbf{f}^+$ if $L_{\mathbf{f}^{+}}h(\mathbf{p})=L^2_{\mathbf{f}^{+}}h(\mathbf{p})=0,L^3_{\mathbf{f}^{+}}h(\mathbf{p})\neq 0$ and $\{ dh(\mathbf{p}) , d(L_{\mathbf{f}^{+}}h)(\mathbf{p}), d(L^2_{\mathbf{f}^{+}}h)(\mathbf{p}) \}$ is a linearly independent set. We define the sets of tangential singularities $S^+=\{\mathbf{p}\in \Sigma \, | \, f_3^+(\mathbf{p})=0\}$ and $S^-=\{\mathbf{p}\in \Sigma \, | \, f_3^-(\mathbf{p})=0\}$.

In $\R^3$, through a generic cusp singularity emanate two branches of fold singularities, see Figure \ref{fig cusp-fold e two-fold inicio}. In one branch it appears visible fold singularities and in the other one invisible fold singularities. 

When $\mathbf{p}$ is a fold, cusp singularity of both smooth vector fields (i.e., $\mathbf{p}$ is a double tangency point) we say that $\mathbf{p}$ is a \textit{two-fold singularity}, \textit{two-cusp singularity}, respectively. When $\mathbf{p}$ is a cusp singularity for one smooth vector field and a fold singularity for the other one, we say that $\mathbf{p}$ is a \textit{cusp-fold singularity}, see Figure \ref{fig cusp-fold e two-fold inicio}.

\begin{figure}[h]
\begin{center} 
\includegraphics[scale=0.5]{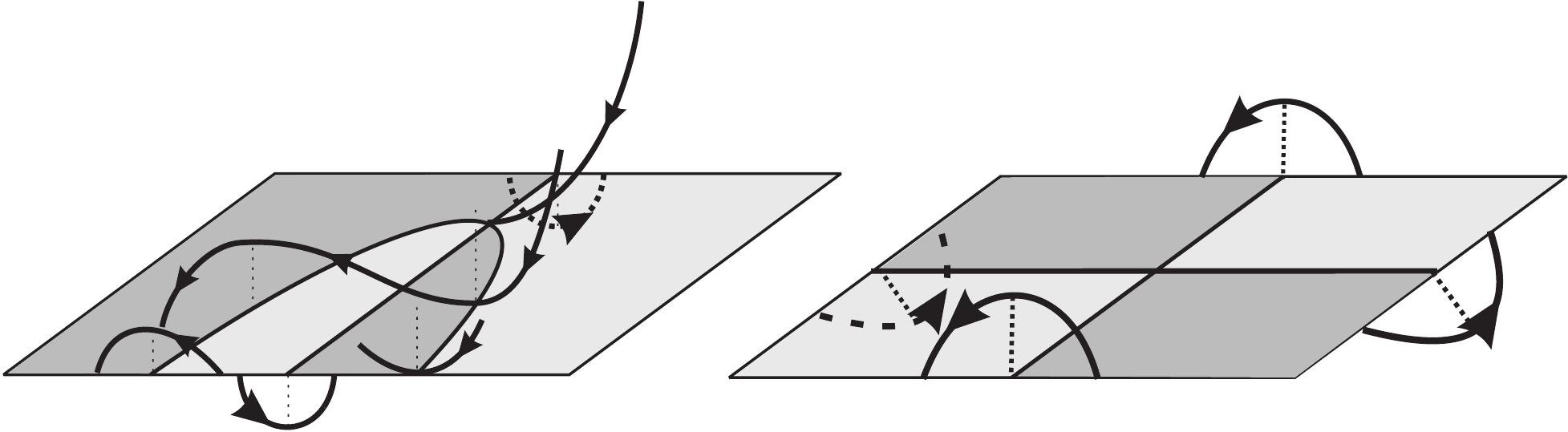}
\caption{On the left it appears a cusp-fold singularity and on the right a two-fold singularity.} \label{fig cusp-fold e two-fold inicio}
\end{center}
\end{figure}

\section{The Boost Converter}\label{section analise conversor}

The behaviour of a dc-dc boost converter, considering the ideal case, can be studied using the circuit topology depicted in Figure \ref{fig_boost}. In our approach, the dynamical system modelling the boost converter with a SMC (Sliding Mode Control) and Washout filter, in continuous conduction mode, is given by
\begin{align}
L\dfrac{di_L}{dt}&=V_{in}-uv_C\label{eq_boost1}\\
C\dfrac{dv_C}{dt}&=ui_L-\dfrac{v_C}{R}\label{eq_boost2}\\
\dfrac{dz_F}{dt}&=\omega_F(i_L-z_F),\label{eq_boost3}
\end{align}
where $v_C$ and $i_L$ are the instantaneous capacitor voltage and the inductor current, respectively. The input voltage is assigned as $V_{in}$, $R$ is the equivalent load resistance, $C$ and $L$ are the circuit capacitor and inductor, respectively. The inductor current $i_L$ passes through a washout filter and a new variable $z_F$ is obtained by \eqref{eq_boost3}, whereas the cut-off frequency of the filter is denoted by $\omega_F$ (see \cite{Ponce-Pagano2}).

The control law is defined as $u=\frac{1}{2}(1+\text{sign}[h])$ such that $u=1$ implies that the key $S$, in Figure \ref{fig_boost}, is off and $u=0$ implies that the key $S$ is on, whereas the planar switching surface is chosen as
\begin{equation*} 
h(i_L,v_C,z_F) = v_C - V_{ref} + K(i_L-z_F),\label{eq:h}
\end{equation*}
where $K>0$ is the control parameter to be adequately tuned and $V_{ref}> V_{in}$ is the reference voltage. The goal of this procedure is to increase the input voltage $V_{in}$ until the value of reference $V_{ref}$.

\begin{figure}[h]
\begin{center} 
\includegraphics[scale=0.55]{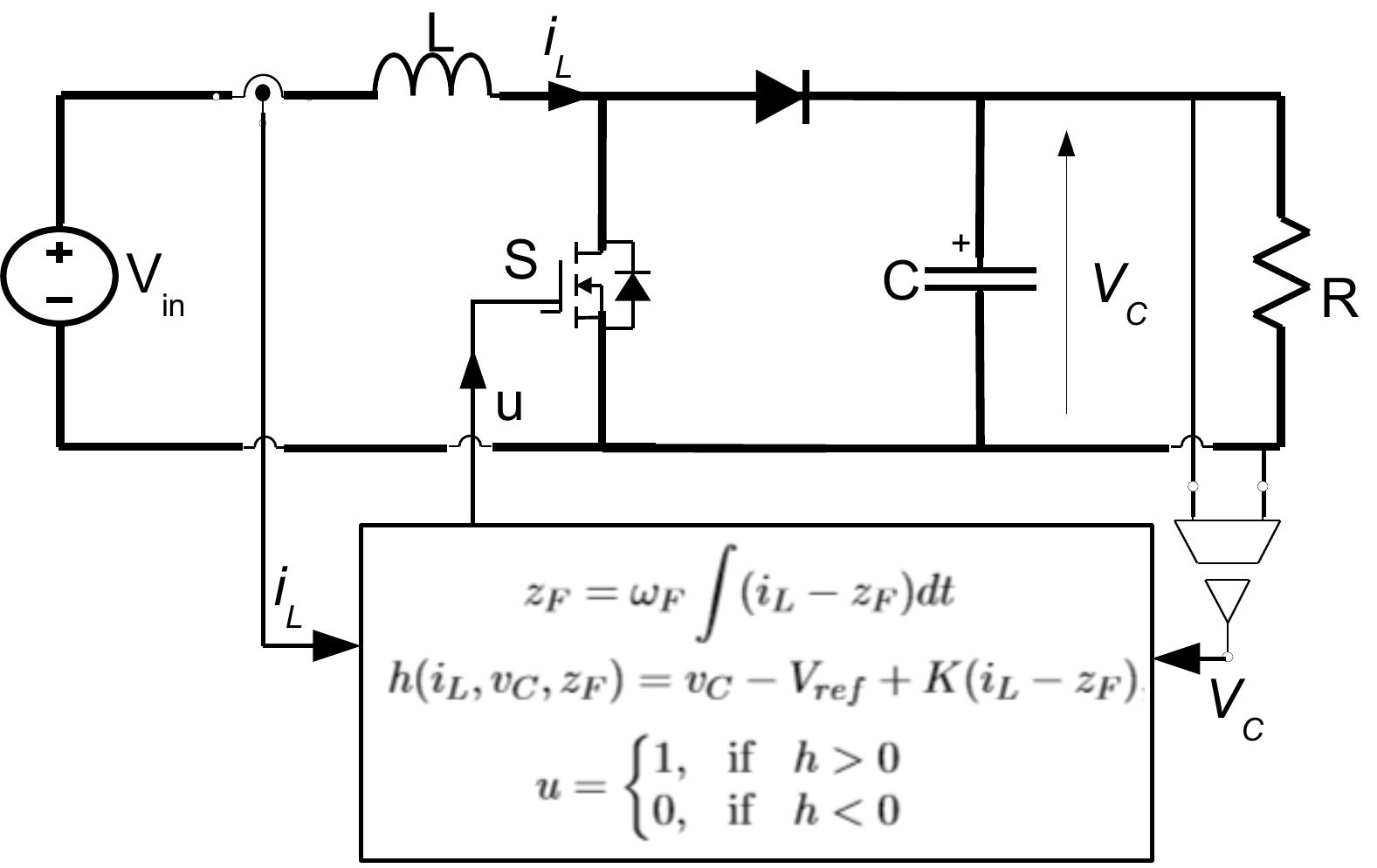}
\caption{Boost Converter with SMC-Washout.} \label{fig_boost}
\end{center}
\end{figure}

Equations \eqref{eq_boost1}-\eqref{eq_boost3} can be normalized applying the following change of variables:
$i_L=V_{in}\sqrt{\frac{C}{L}}x$, $v_C=V_{in}y$, $z_F=i_L+\frac{v_C-V_{ref}-V_{in}z}{K}$; and in the time $t=\sqrt{CL}\tau$. Defining the new parameters: $a=\frac{1}{R}\sqrt{\frac{L}{C}}$, $k=K\sqrt{\frac{C}{L}}$, $\omega=\omega_F\sqrt{LC}$ and $y_r=\frac{V_{ref}}{V_{in}}$; the dimensionless model is given by
\begin{equation}\label{eq conversor geral}
\begin{array}{ccl}
\dot{x}&=&1-uy\\
\dot{y}&=&ux-ay\\
\dot{z}&=&u(x-ky)+(\omega-a)y-\omega z+k - \omega y_r,
\end{array}\end{equation}
where $(x,y,z)\in D\subset\mathbb{R}^3$ are the independent variables and the parameters are $\omega\in(0,1]$, $y_r>1$, $k>0$ and $a>0$ (the dot $``\cdot"$ indicates $\frac{d}{d\tau}$). We stress that $x>0$ is the  normalized inductor current, $y\geq 0$ is the normalized output voltage and $z\in\mathbb{R}$ depends on the filtered current. 

For the normalized system \eqref{eq conversor geral}, the control law is defined as 
\begin{equation}
u=\frac{1}{2}(1+\text{sign}[z])\label{control}
\end{equation}
and switching manifold as $\Sigma=\{(x,y,z)\in\mathbb{R}^3: z=0\}$. The system \eqref{eq conversor geral} with control law \eqref{control} can be represented by piecewise smooth dynamical system $(\dot{x},\dot{y},\dot{z})=\mathbf{f}(x,y,z)$ with
\begin{equation}\label{eq conversor escrito PSVFgeral}
\mathbf{f}(\mathbf{x})=\left\{\begin{array}{c}
\mathbf{f}^{+}=(-1-y,x-ay,f_3^+)\;\;\;\text{if}\;\;z>0\\
\mathbf{f}^{-}=(1,-ay,f_3^-)\;\;\;\;\;\;\;\;\;\;\;\;\;\;\;\;\text{if}\;\;z<0
\end{array}\right.,
\end{equation}  
where $\mathbf{x}=(x,y,z)$ and
\begin{align*}
f_3^+(x,y,z)&=x+(\omega-a-k)y- \omega z+k-\omega y_r,\\
f_3^-(x,y,z)&=(\omega-a)y- \omega z+k-\omega y_r.
\end{align*} 

\begin{remark}
In the sequel we consider $0<a<2$. This is a coherent physical hypothesis and it is enough to produce the desired behavior. 
\end{remark}

\subsection{Tangential Singularities}

The tangential sets of $\mathbf{f}^{+}$ and $\mathbf{f}^{-}$ are given, respectively, by the \textit{straight lines}:
\begin{align*}
S^{+}&=\left\{(x,y,0)\in\Sigma : x=(a+k-\omega)y-k+\omega y_r\right\},\\
S^{-}&=\left\{(x,y,0)\in\Sigma : y=\dfrac{k-\omega y_r}{a- \omega},\;\text{for}\;a\neq \omega \right\}.
\end{align*}
The next result summarizes the possibilities of tangential singularities according to the parameters $a$, $y_r$, $k$ and $\omega$.

A straightforward calculation shows that the point $\mathbf{p}_c=(x_c,y_c,0)$ with
\begin{align*}
x_c&= \dfrac{\omega(y_r-1) + a(1 + (a + k - \omega)(k - \omega y_r))}{(a+k-\omega)(k-\omega)+1}\\
y_c&=\dfrac{(a+k-\omega)(k-\omega y_r)+1}{(a+k-\omega)(k-\omega)+1},
\end{align*}
is a \textit{cusp singularity}, since $L^2_{\mathbf{f}^{+}}h(\mathbf{p}_c)=0$ and the third Lie derivative evaluated in cusp point is given by $L^3_{\mathbf{f}^{+}}h(\mathbf{p}_c)=\omega(1-y_r)<0$, i.e., the trajectory of $\mathbf{f}^+$ passing through the cusp point $\mathbf{p}_c$ departs from $\Sigma$. The point $\mathbf{p}_c$ separates $S^{+}$ into two branches of fold singularities. The \textit{branch of visible fold singularities} for $y<y_c$ and the \textit{branch of invisible fold singularities} for $y>y_c$. 

Since $L^2_{\mathbf{f}^{-}}h(\mathbf{p})=a(k-\omega y_r)$ for all $\mathbf{p}\in S^{-}$ we get that all points in $S^{-}$ are \textit{invisible fold singularities} if $k>\omega y_r$, or \textit{visible fold singularities} if $k<\omega y_r$. 

The double tangency point, $\mathbf{p}_t$, is given by $S^{+} \cap S^{-}$, i.e., 
\begin{equation}\label{ponto-dobra-dobra}
\mathbf{p}_t = \Big( \dfrac{k (k- y_r \omega)}{a- \omega} ,  \dfrac{k-\omega y_r}{a- \omega} , 0 \Big).
\end{equation}
The point $\mathbf{p}_t$ is a two-fold singularity if $a\neq a_c(k)$ or a fold-cusp singularity if $a= a_c(k)$, for all $k\neq\omega y_r$, where
\begin{align}
a_c(k)&=\dfrac{1}{2(k - \omega y_r)}\Big[-1+\omega(k-\omega y_r)\label{ac}\\
&+\sqrt{1 + (k - \omega y_r)(2\omega + (4 + (\omega-2k)^2)(k - \omega y_r))}\Big]\nonumber
\end{align}

Table \ref{tabela-tipos-tangencia} shows the kinds of double tangency points according to the parameters $(a,k)$.

\begin{table}[!htpb]
	\centering
\begin{tabular}{|l|l|}\hline
\textbf{Kind of tangency} & \textbf{Region on the plane $(a,k)$} \\
	\hline \hline
Two-Fold Visible-Invisible & $a_c(k)<a<\omega$ for $k<\omega y_r$ \\
\hline
Two-Fold Visible-Visible & $a<a_c(k)$ for $k<\omega y_r$\\
	\hline
Two-Fold Invisible-Invisible & $\omega<a<a_c(k)$ for $k>\omega y_r$\\
	\hline
Two-Fold Invisible-Visible & $a>a_c(k)$ for $k>\omega y_r$ \\
	\hline
Fold Invisible-Cusp & $a=a_c(k)$ for $k>\omega y_r$\\
	\hline
Fold Visible-Cusp & $a=a_c(k)$ for $k<\omega y_r$\\
\hline
\end{tabular}
\vs
\caption{Kinds of tangential points according to the parameters $(a,k)$.}
\label{tabela-tipos-tangencia}
\end{table}

\begin{remark}
	Observe that, for practical reasons, $x> 0$. As consequence, in $\mathbf{p}_t$, either $a>\omega$ and $k>\omega y_r$ or $a<\omega$ and $k<\omega y_r$.
\end{remark}

\subsection{Dynamics of the sliding vector field}

The sliding vector field is calculated from the equation \eqref{eq campo filippov}, resulting in
\begin{equation}
\widehat{\mathbf{f}}^s(\mathbf{x})=\dfrac{1}{x-ky}\begin{bmatrix}
x - ay^2 + \omega y(y-y_r - z)\\
-k(x - ay^2) - \omega x (y - yr - z)\\
0
\end{bmatrix},\label{Fs}
\end{equation}
whose equilibrium point is
\begin{equation}
\mathbf{q}=(ay_r^2,y_r,0).\label{PE}
\end{equation}

This is the operation point of the boost converter with a sliding mode controller and Washout filter. This filter is responsible by the elimination of the output tension dependence in relation to the parameter of the resistive charge $a$. In this way, after a perturbation on $a$, the output tension keeps the desired value $y_r$. 

  At this moment, we have to analyze if the pseudo-equilibrium $\mathbf{q}$ is either real or virtual. In other words, we explicit conditions, under the parameters, to get $\mathbf{q}\in \Sigma^s$. In fact, these conditions are given by 
\begin{align*}
L_{\mathbf{f}^{-}}h(\mathbf{q})&=k-ay_r,\\
L_{\mathbf{f}^{+}}h(\mathbf{q})&=-(y_r-1)(k - ay_r).
\end{align*}  
  We remember that the parameter $y_r>1$. Therefore, if $k-ay_r>0$ (resp., $k-ay_r<0$) then $\mathbf{q}\in \Sigma^s$ (resp., $\mathbf{q}\in \Sigma^e$) and if $k-ay_r=0$ then $\mathbf{q}$ coincides with the double tangency point $\mathbf{p}_t$ given in \eqref{ponto-dobra-dobra}.

In the sliding mode control it is necessary that the pseudo-equilibrium (i.e., the operation point) remains in the sliding region $\Sigma^s$. So, the control parameter $k$ satisfies:
\begin{equation}
k>ay_r.\label{condicao_k_1}
\end{equation}
Moreover, the pseudo-equilibrium must be stable and without limit cycle around it. 
 
  In order to analyze the stability of the pseudo-equilibrium $\mathbf{q}$, we use the sliding vector field  $\mathbf{f}^{s}$ calculated in accordance with \eqref{equacao campo normalizado}. In this case, we obtain
 \begin{equation}
 \mathbf{f}^{s}(x,y)=\begin{bmatrix}
-x + ay^2 - \omega y(y-y_r)\\
k(x - ay^2) + \omega x (y - yr)
\end{bmatrix}.\label{Fds}
 \end{equation}
  The projection of pseudo-equilibrium $\mathbf{q}$ in the switching manifold $\Sigma$ is the point $\mathbf{q}_s=(ay_r^2,y_r)$. This point is an equilibrium of $\mathbf{f}^{s}$ and its stability can be extended to pseudo-equilibrium $\mathbf{q}$ since satisfied the condition \eqref{condicao_k_1}.

The Jacobian matrix of the normalized sliding vector field \eqref{Fds} evaluated at the point $\mathbf{q}_s$ is given by
\[J(\mathbf{q}_s)=\left(
\begin{array}{cc}
-1 & (2a - \omega)y_r \\
k & ay_r(\omega y_r-2k) \\
\end{array}
\right).
\]
Then, we calculated the determinant and trace of $J(\mathbf{q}_s)$, getting
\begin{align*}
\text{Det}[J(\mathbf{q}_s)]   &=\omega y_r (k - a y_r),\\
\text{Tr}[J(\mathbf{q}_s)]     &=-1 + ay_r(\omega y_r-2k).
\end{align*}
Imposing the condition \eqref{condicao_k_1} on the parameter $k$, then $\text{Det}[J(\mathbf{q}_s)]>0$. Therefore, the pseudo-equilibrium $\mathbf{q}$, when in $\Sigma^s$, can be a pseudo-node or pseudo-focus, stable or unstable. In this case it will be stable if, and only if, $ \text{Tr}[J(\mathbf{q}_s)]<0$, i.e., $k$ must be chosen such that it satisfies, in addition of inequality \eqref{condicao_k_1}, inequality
\begin{equation}
k>\dfrac{a \omega y_r^2-1}{2ay_r}.\label{condicao_k_2}
\end{equation}

\begin{remark}
If $k<ay_r$ the pseudo-equilibrium $\mathbf{q}$ is on the scape region $\Sigma^e$ and is a pseudo-saddle, because $\text{Det}[J^e(\mathbf{q}_s)]=\text{Det}[J(\mathbf{q}_s)]<0$, where $J^e$ is the Jacobian matrix of the normalized sliding vector field definite in $\Sigma^e$ as $\mathbf{f}^{e}(x,y)=-\mathbf{f}^{s}(x,y)$.
\end{remark}

 More precisely, to distinct if $\mathbf{q}$ is a pseudo-focus or a pseudo-node we have to analyze the signal of discriminant $\Delta$ of the characteristic polynomial of $J(\mathbf{q}_s)$. Explicitly,
\begin{align*}
\Delta    &= \text{Tr}[J(\mathbf{q}_s)]^2 - 4\text{Det}[J(\mathbf{q}_s)]\nonumber\\
&= 4a^2y_r^2k^2-4y_r(\omega+a(a\omega y_r^2-1))k+(1+a\omega y_r^2)^2.\label{discriminate-p1}
\end{align*}
This expression is a polynomial of degree two in the variable $k$. The solutions of $\Delta=0$ are given by: 
\begin{equation}
k_{\pm} = k_{H}+\dfrac{\omega\pm\sqrt{\omega(1+2a^2y_r^2)(\omega-2a)}}{2a^2y_r}\label{Kmaismenos}
,\end{equation}
where 
\begin{equation}
k_{H}=\dfrac{a \omega y_r^2-1}{2ay_r}.\label{kHopf}
\end{equation}
Note that $ k>k_{H}$ satisfies the stability condition \eqref{condicao_k_2} and $k=k_{H}$ implies $\text{Tr}[J(\mathbf{q}_s)]=0$. Furthermore, the roots $k_{\pm}$ of polynomial $\Delta=0$ exist only for $a\leq\dfrac{\omega}{2}$, otherwise we will have $\Delta>0$ for all $k$.

	 In the following, in Table \ref{tabela-estabilidade_PE}, we summarize this results about the dynamics of the sliding vector field at pseudo-equilibrium point $\mathbf{q}$. At this stability result we consider $y_r\geq\frac{2\sqrt{2}}{\omega}$, otherwise $\mathbf{q}$ is always stable for all $k>ay_r$. This condition on the parameter $y_r$ assures us the existence of the \textit{Bogdanov-Takens} bifurcation (BT) in the points $(a_-,y_ra_-)$ and $(a_+,y_ra_+)$, of plan $(a,k)$ (see Figure \ref{fig_set_bif}), such that
	 	 \begin{equation}
a_{\pm}=\dfrac{1}{4y_r}\left(\omega y_r \pm \sqrt{\omega^2y_r^2-8}\right).\label{eq_amais_menos}
\end{equation}
 
\begin{table}[!htpb]
	\centering
	\begin{tabular}{|l|l|}\hline
		\textbf{Kind of dynamics} & \textbf{Conditions under the parameters $(a,k)$} \\
		\hline \hline
pseudo-saddle  & $k<ay_r$ \\
		\hline
 stable pseudo-node & $k>ay_r$ and $a\geq\frac{\omega}{2}$, or \\
  & $k>k_+$ and $ a<\frac{\omega}{2}$, or \\
  & $ay_r<k<k_-$ and \\&$a_+< a<\frac{\omega}{2}\cup 0<a<a_-$ \\
\hline
 unstable pseudo-node & $ay_r<k<k_-$ and $a_-< a<a_+$ \\
		\hline
stable pseudo-focus  & $k_-<k<k_+$ and \\&$ a_+<a<\frac{\omega}{2}\cup 0<a<a_-$, or\\
& $k_H<k<k_+$ and $a_-< a<a_+$\\
				\hline
unstable pseudo-focus & $k_-<k<k_H$ and $a_-< a<a_+$\\
		\hline
	\end{tabular}\vs
	\caption{Kinds of dynamics of the sliding vector field at pseudo-equilibrium point $\mathbf{q}$, according the parameters $(a,k)$.}
	\label{tabela-estabilidade_PE}
\end{table}

Now we study the Hopf and Homoclinic bifurcations for the sliding vector field from the boost converter system with SMC-Washout. 
 

\section{A Hopf Bifurcation followed by a Homoclinic Loop at the Two-Fold Singularity}\label{secao bifurcacoes}

In this section we analyze two bifurcations occurring in the sliding vector field \eqref{Fs}. First a Hopf Bifurcation takes place giving rise to a limit cycle. After, this limit cycle increases and then collides with the two-fold point, which behaves like a saddle. This meeting produces a homoclinic loop destroying the limit cycle.

\subsection{The Hopf Bifurcation}\label{subsecao bifurcacao de hopf}

In the previous section we proved that  $\mathbf{q}$ is an unstable focus when $k_-<k<k_H$ and a stable one when $k_H<k<k_+$, since $a_-< a<a_+$. Then we can state the following result:

\begin{proposition}\label{proposicao bif hopf}
	If  $k=k_H$ and $a\in (a_-,a_+)$, where $a_{\pm}$ are given in \eqref{eq_amais_menos} and $k_H$ is given in \eqref{kHopf}, then a subcritical Hopf bifurcation
	occurs at $\mathbf{q}=(a y_r^2,y_r,0)$ in the sliding vector field \eqref{Fs}.
\end{proposition}
\begin{proof}
For this proof, we consider the planar sliding vector field \eqref{Fds} that is topologically equivalent to \eqref{Fs} in $\Sigma^s$ and the projection of pseudo-equilibrium $\mathbf{q}$ in $\Sigma$ given by $\mathbf{q}_s=(a y_r^2,y_r)$.

The necessary conditions to get a Hopf bifurcation are satisfy for $k=k_H$:
\begin{align*}
\text{Det}[J(\mathbf{q}_s)]\Big|_{k=k_H}&=\dfrac{\omega(-1+\omega y_r^2a-2y_r^2a^2)}{2a}>0\\
\text{Tr}[J(\mathbf{q}_s)]\Big|_{k=k_H}&=0\\
\dfrac{d\text{Tr}[J(\mathbf{q}_s)]}{dk}\Big|_{k=k_H}&=-2ay_r\neq 0;
\end{align*}
since $a_-<a<a_+$. Then let us consider the system 
\begin{equation}\label{eq transladado}
\begin{array}{ll}
\dot{u}    &=u + v (\omega (v + y_r) - a (v + 2 y_r))\\ \\
\dot{v}  &=-v \omega (u + a y_r^2) + k (-u + a v (v + 2 y_r)),
\end{array}
\end{equation}
obtained from a translation of \eqref{Fds} in such a way that $\mathbf{q}_s$ is translated to the origin.
		
According to \cite{A-II}, pg 253, if the number 
\begin{equation*}
\begin{array}{ll}
\sigma    &=-\frac{1} {y_r (2 a-\omega) \sqrt{\frac{\omega^3 (a y_r^2 (\omega-2 a)-1)^3}{a^3}}}  (3 \sqrt{2} \pi  \omega (-2 a^2 k^2 y_r\\\\
&+a (2 k^3+k^2 \omega y_r+k-2 \omega y_r)+\omega (k+\omega y_r)))
\end{array}
\end{equation*}
is not null, then a Hopf bifurcation occurs at the origin in the planar analytic system \eqref{eq transladado}. Moreover, we have $\sigma >0$ for all $a\in (a_-,a_+)$ and, thus, a subcritical Hopf bifurcation occurs when $k=k_H$. Therefore, an unique unstable limit cycle bifurcates from the point $\mathbf{q}$ in the sliding vector field \eqref{Fs} (see Figures \ref{fig_uLC}-\ref{fig_uLC_3D}). So the result is proved.  
\end{proof}

 Figures \ref{fig_uLC} and \ref{fig_uLC_3D} show the phase portrait of \eqref{Fds} and a simulation of the behavior of the boost converter with the SMC-Washout given by \eqref{eq conversor escrito PSVFgeral}, respectively. On both we note the existence of a limit cycle $C\subset \Sigma^s$ around the stable focus $\mathbf{q}\in \Sigma^s$. The red closed curve, the blue point and the green point represent the unstable limit cycle (uLC) $C$, the pseudo-equilibrium $\mathbf{q}$ and the two-fold point $\mathbf{p}_t$, respectively. The parameter values used in the simulation are: $\omega=1$, $y_r=4$, $a=0.2$ and $k=1.5$.
 
\begin{figure}[h]
\begin{center}
\subfigure[Projection on $\Sigma$.]{
\includegraphics[scale=0.5]{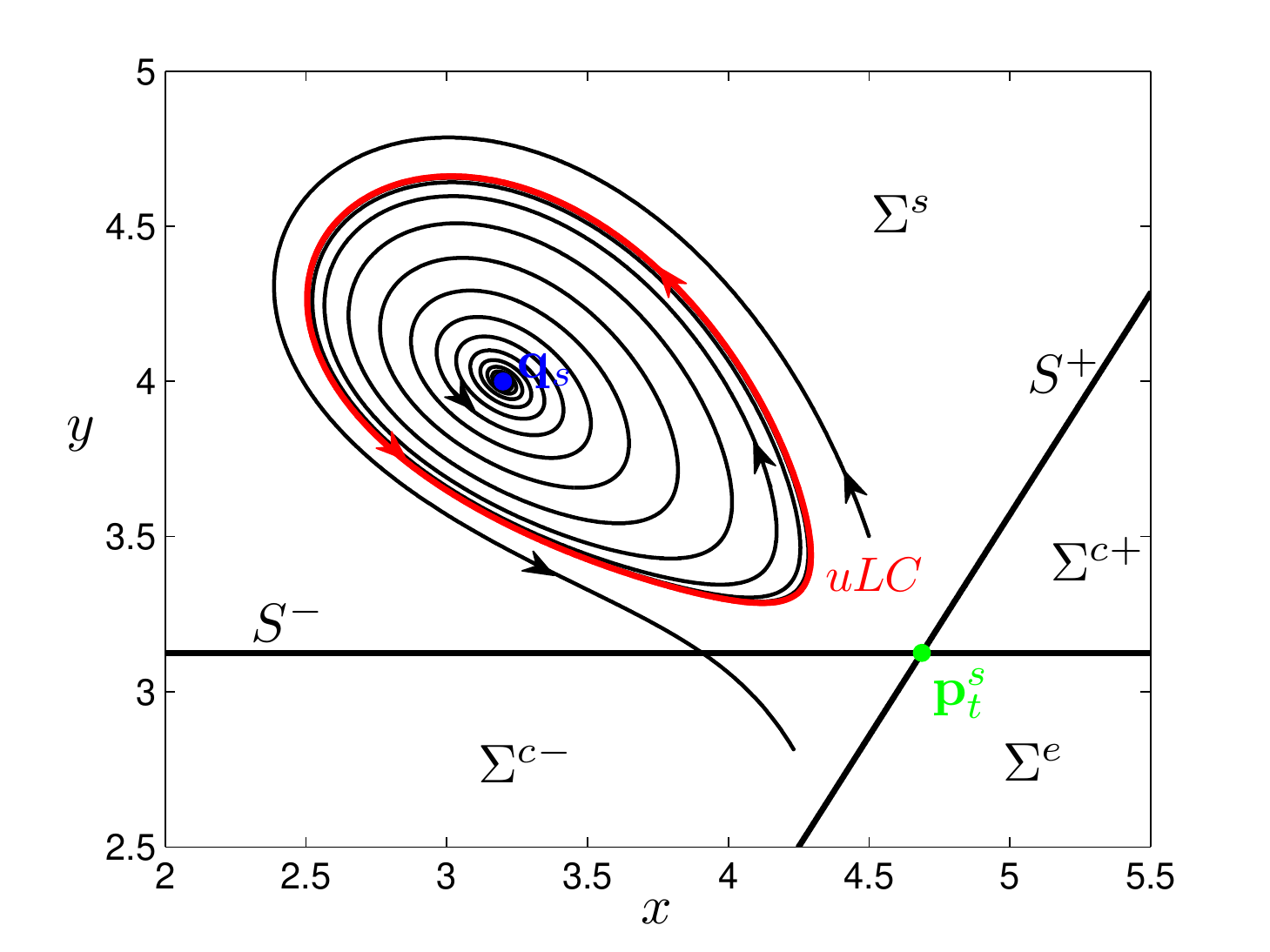}\label{fig_uLC}}
\subfigure[Simulation result of Boost converter with SMC-Washout.]{
\includegraphics[scale=0.6]{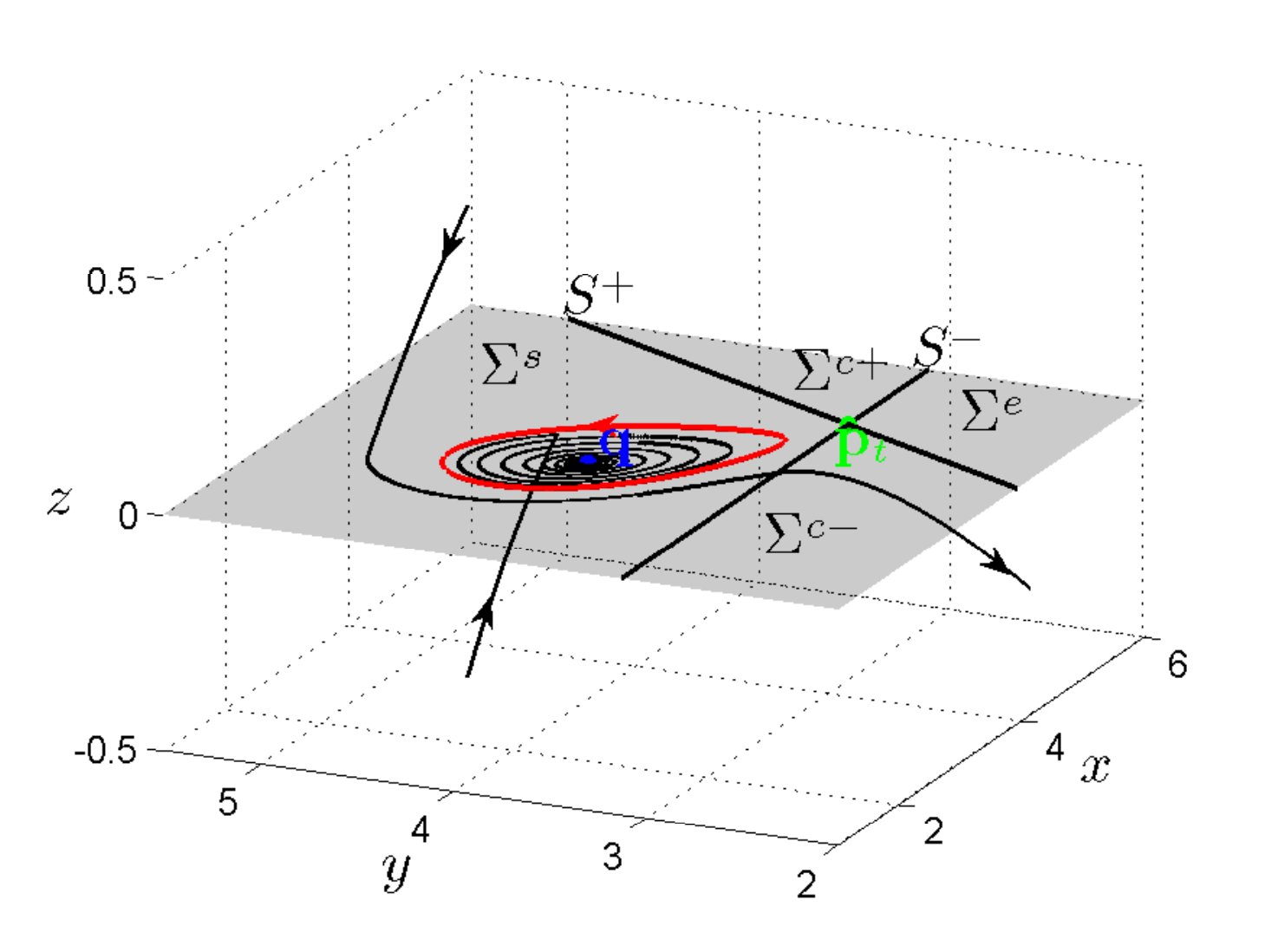}\label{fig_uLC_3D}}
\caption{Unstable limit cycle in the nonsmooth vector field \eqref{eq conversor escrito PSVFgeral}.}
\end{center}
      \end{figure}
      

\subsection{The Homoclinic Loop Bifurcation}\label{subsecao bifurcacao loop}

The unique limit cycle $C$ emerged from the Hopf Bifurcation of Subsection \ref{subsecao bifurcacao de hopf} increases until to collapse with the two-fold point. This is the content of the next proposition.

\begin{proposition}\label{proposicao-bif-homoclinic}
The limit cycle $C$ emerged from the Hopf Bifurcation in Proposition \ref{proposicao bif hopf} increases until to collapse with the visible-invisible two-fold point $\mathbf{p}_t$ giving rise to a homoclinic loop.
\end{proposition}

\begin{proof}
According to Table \ref{tabela-tipos-tangencia}, the double tangency point  $\mathbf{p}_t$ is a visible-invisible two-fold singularity whenever $a_c(k)<a<\omega$ and $k<\omega y_r$. Note that the Homoclinic bifurcation curve in plan $(a, k)$ is contained in the quadrant $a_-<a<a_+$ and $y_ra_-<k<y_ra_+$ (see Firgure \ref{fig_set_bif}). As $a_+<\frac{\omega}{2}$ and $a_->a_c$, then the double tangency point is classified as visible-invisible two-fold when the Homoclinic loop bifurcation occurs.

Moreover, the projection of the point $\mathbf{p}_t$ on the switching manifold $\Sigma$ is the point $\mathbf{p}_t^s=(ky_t,y_t)$, where $y_t>0$ is the coordinate $y$ of the double tangency point $\mathbf{p}_t$ given in \eqref{ponto-dobra-dobra}. The point $\mathbf{p}_t^s$ is an equilibrium of the planar normalized sliding vector field $\mathbf{f}^s$, whose dynamics in the your neighborhood in $\Sigma^s$ is saddle type whenever the $\mathbf{q}\in\Sigma^s$, because
\begin{equation*}
\text{Det}[J(\mathbf{p}_t^s)]=-\omega (k-ay_r)y_t<0 
\end{equation*}  
  for $k>ay_r$. Therefore it is natural that the homoclinic loop passes through this point. 
  
Since the two coordinates of the system \eqref{Fds} have no roots in common and the cycle emerged from Proposition \ref{proposicao bif hopf} is unique, we are able to use the \textit{Perko's Planar Termination Principle} (see \cite{Perko-1,Perko-2}) in order to ensure that $C$ collapses with the visible-invisible two-fold point $\mathbf{p}_t$ giving rise to a homoclinic loop. See Figures \ref{fig_HC}-\ref{fig_HC_3D}.

\end{proof}

Figures \ref{fig_HC} and  \ref{fig_HC_3D} illustrate the phase portrait of \eqref{Fds} and the simulations of the boost converter with SMC-Washout given by \eqref{eq conversor escrito PSVFgeral}, respectively. On both we note the homoclinic loop (purple curve) passing through $\mathbf{p}_t$. In Figure \ref{fig_HC} we observe that  $\mathbf{p}_t$ is a saddle equilibrium and is the closing point of the orbit homoclinic. In Figure \ref{fig_HC_3D} we observe that it is a visible-invisible two-fold singularity of the model. The parameter values used in the simulation are: $\omega=1$, $y_r=4$, $a=0.2$ and $k=1.573$.
      
\begin{figure}[h!]
\begin{center}
\subfigure[Projection on $\Sigma$.]{
\includegraphics[scale=0.5]{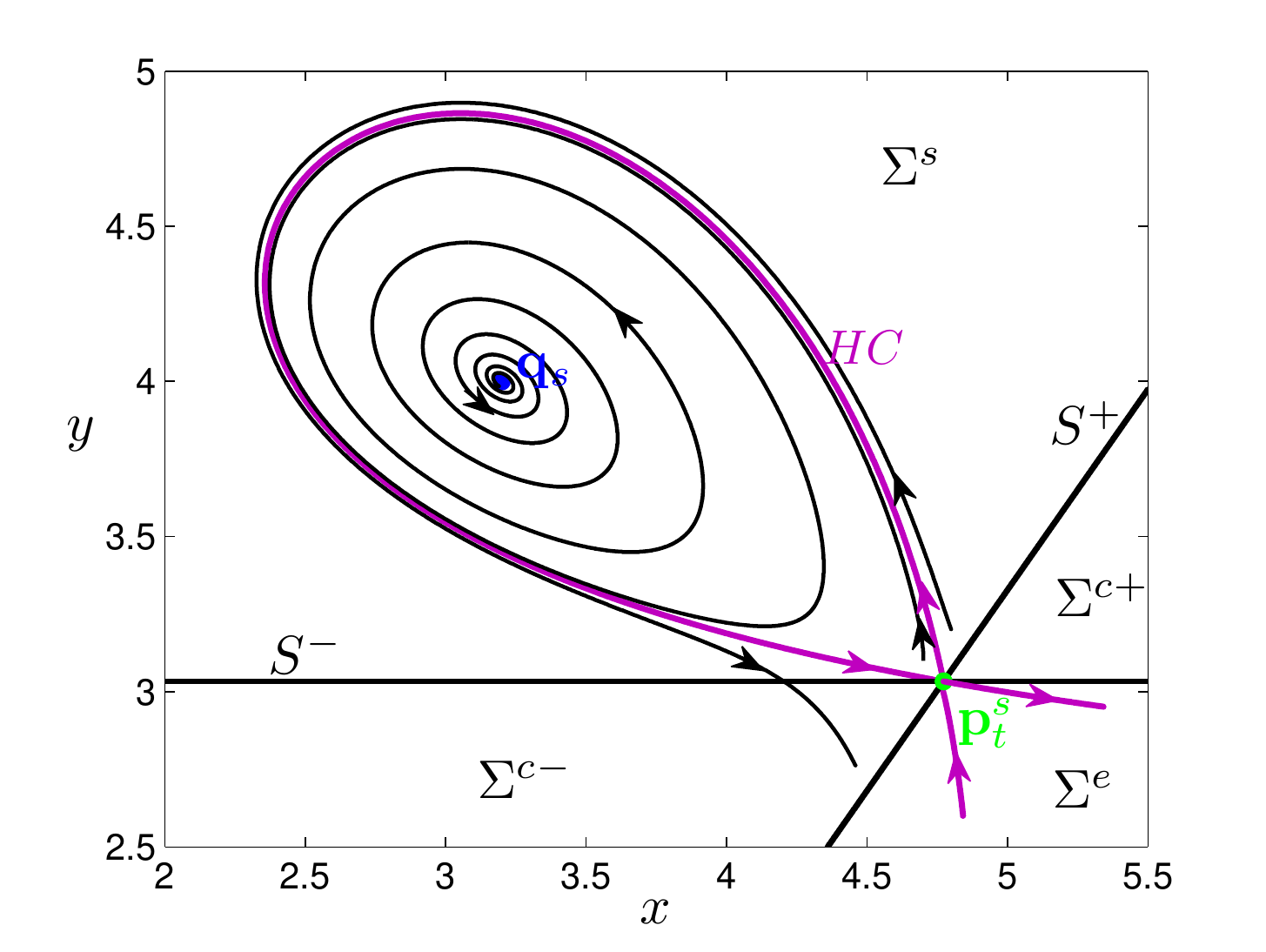}\label{fig_HC}}
\subfigure[Simulation result of Boost converter with SMC-Washout.]{
\includegraphics[scale=0.6]{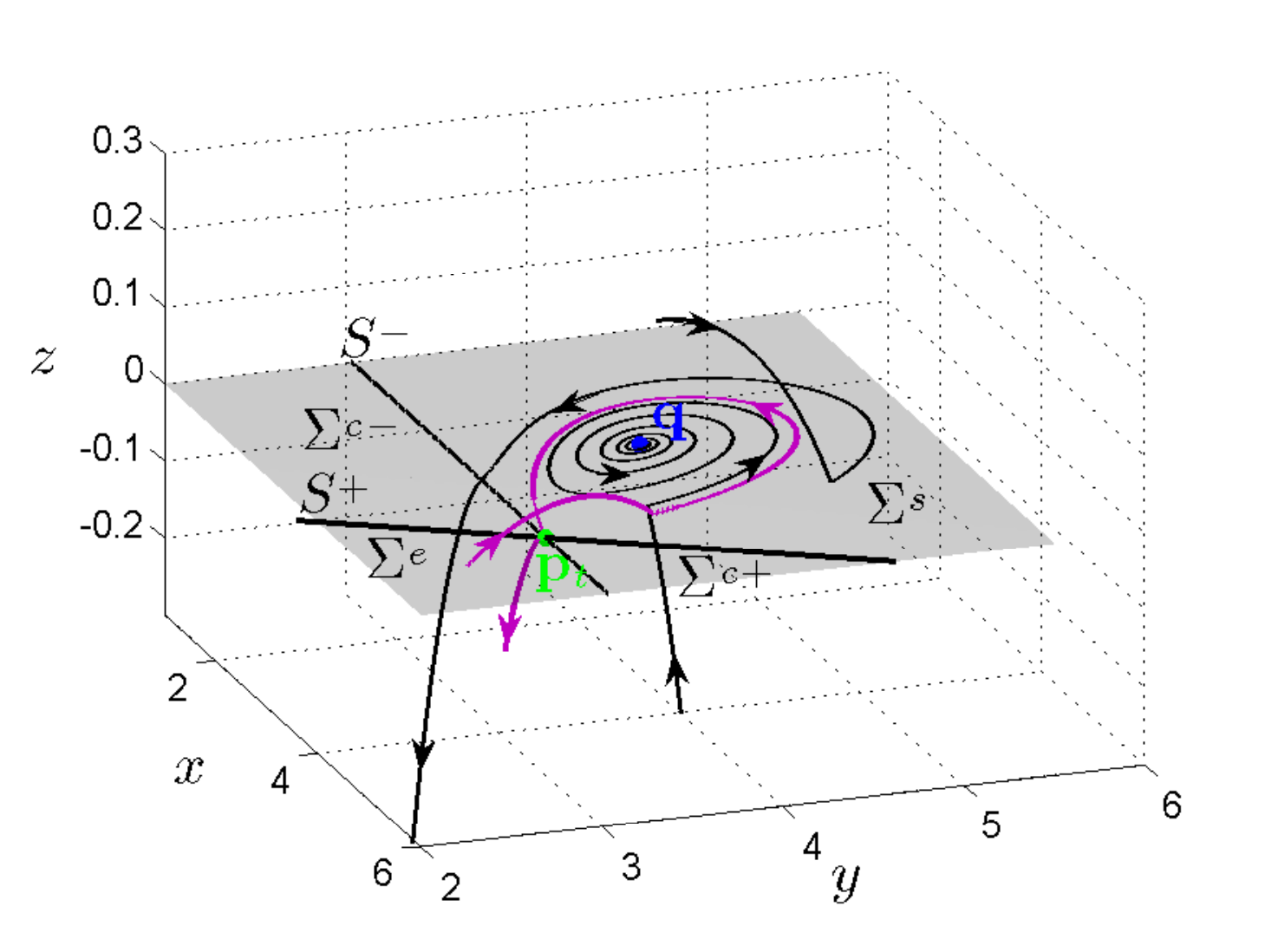}\label{fig_HC_3D}}
\caption{Homoclinic loop at the two-fold point.}
\end{center}
      \end{figure}
      

\begin{remark}
Observe that the homoclinic loop $L_0$ is \textit{simple} since
	\[\sigma_0= \text{Tr}[J(\mathbf{p}_t^s)]= \dfrac{k(2a - \omega)(\omega y_r-k)}{a-\omega} -1\neq 0.\]By a result on page 304 of \cite{A-II}, the homoclinic loop $L_0$ repels all of the trajectories in some inner neighborhood of $L_0$ because the limit cycle that collides with the invariant eigenvectors of $\mathbf{p}_t^s$ (a saddle point) is an unstable limit cycle, that approach continuously to the stable and unstable separatrixs of $\mathbf{p}_t^s$.

Moreover, as expected, from this homoclinic loop can bifurcates only one limit cycle, as proved in a result on page 309 of \cite{A-II}.
\end{remark}

\subsection{Final Remarks}
The unstable limit cycle emerging from the Hopf bifurcation disappears in the Homoclinic loop bifurcation (see the bifurcation diagrams of Figure \ref{fig_diag_bif}). In these diagrams we get: the black curve representing the amplitude of of the unstable limit cycle; the blue straight line representing the coordinates $x$ and $y$ of $\mathbf{q}$, which is an unstable focus in the dashed region and a stable one at the not dashed region; the green dashed curve represents the coordinates $x$ and $y$ of the two-fold singularity $\mathbf{p}_t$; the red point indicates the subcritical Hopf bifurcation ($H_{sub}$), where the unstable limit cycle born; and the purple point indicates the Homoclinic loop bifurcation (HC), where the unstable limit cycle collides with the two-fold singularity $\mathbf{p}_t$ and disappears.
 
 \begin{figure}[h]
 	\begin{center}
 		\subfigure[$(x,k)-$plane]{
 			\includegraphics[scale=0.415]{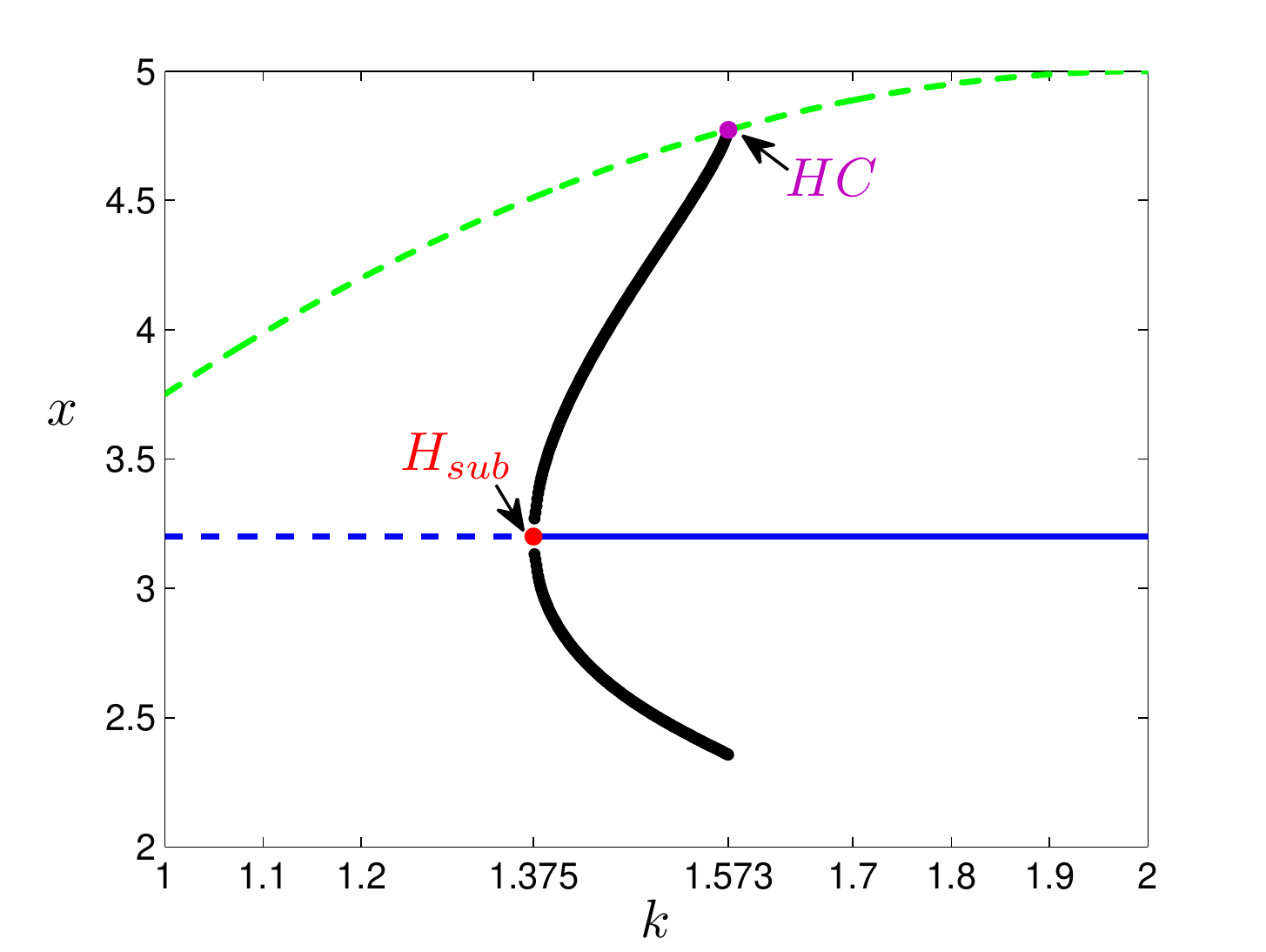}}
 		\subfigure[$(y,k)-$plane]{
 			\includegraphics[scale=0.415]{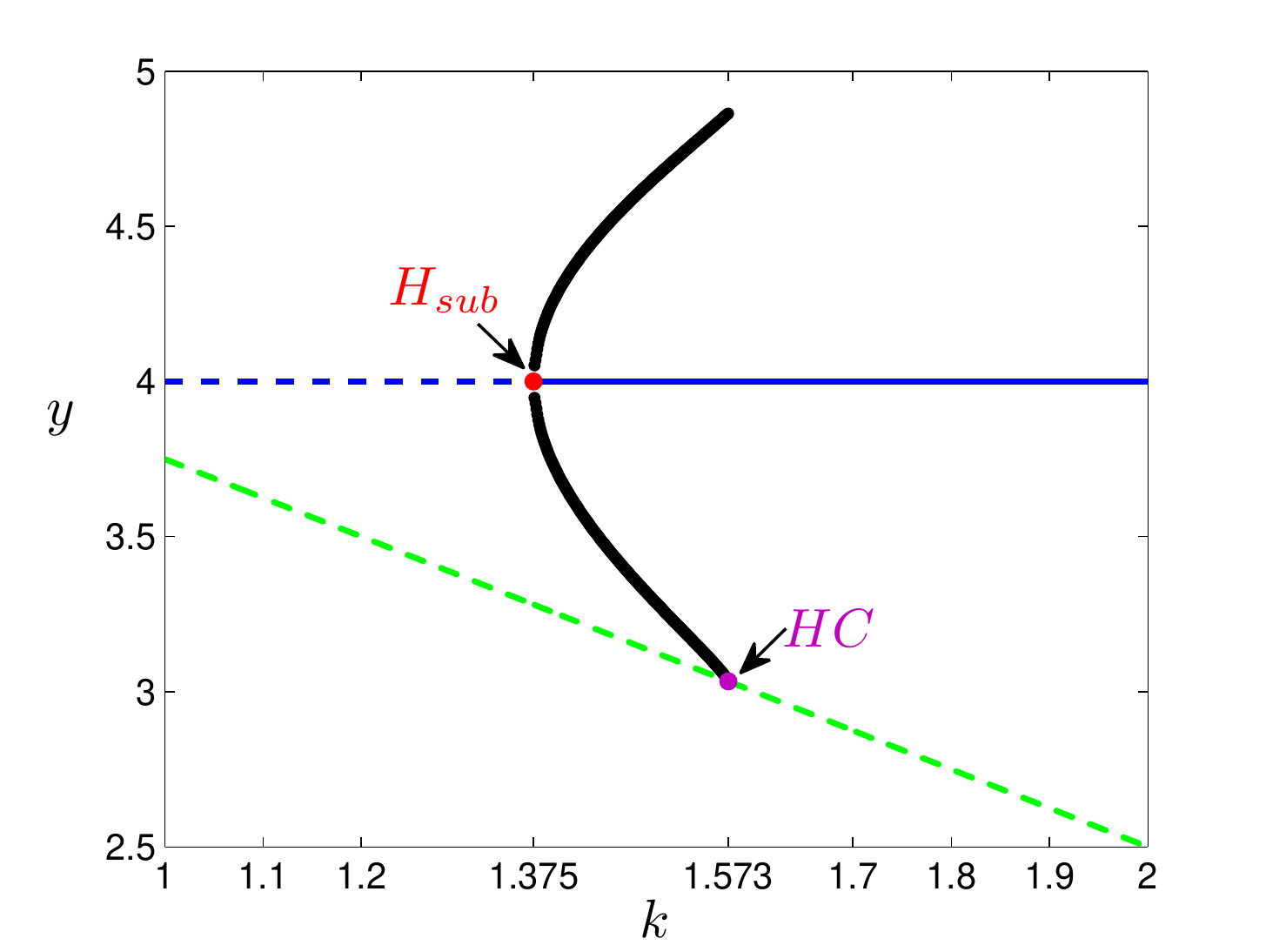}}
 		\caption{Bifurcations Diagrams Hopf and Homoclinic.}\label{fig_diag_bif}
 	\end{center}
 \end{figure}
 \begin{figure}[h]
 	\begin{center}
 		\includegraphics[scale=0.55]{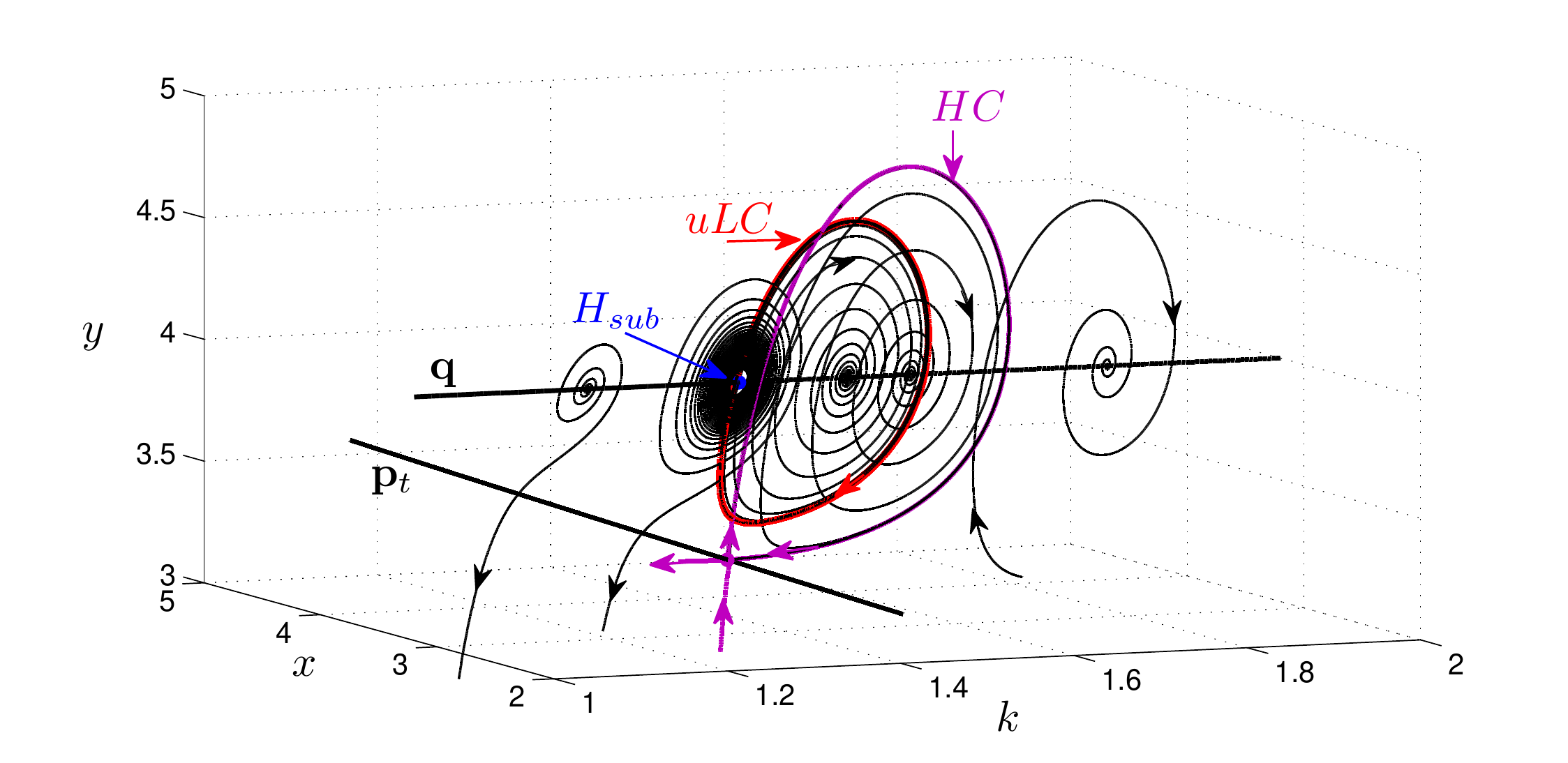}
 		\caption{Bifurcations diagram in $(x,y,k)-$space. }\label{fig_diag_bif_R3}
 	\end{center}
 \end{figure}

 Figure \ref{fig_diag_bif_R3} illustrates the evolution of the pseudo-focus $\mathbf{q}$, of the two-fold singularity $\mathbf{p}_t$ and of the unstable limit cycle (uLC), in relation to the parameter $k$. The pseudo-focus $\mathbf{q}$ is unstable before the subcritical Hopf bifurcation ($H_{sub}$) and stable after it. The unstable limit cycle surrounds $\mathbf{q}$ and there exists when $k\in(1.375,1.573)$. For $k=1.573$ it disappears colliding to $\mathbf{p}_t$ in a homoclinic loop (HC). 
 
 We summarize the dynamics on the diagram of Figure \ref{fig_set_bif}, where we get: the two points BT, given by $(a_-,y_ra_-)$ and $(a_+,y_ra_+)$ with $a_{\pm}$ given in \eqref{eq_amais_menos}, represent two Bogdanov-Takens bifurcations of $\mathbf{q}$; the blue curve, given by $k=k_H(a)$ with $k_H$ as \eqref{kHopf}, indicates the subcritical Hopf bifurcation ($H_{sub}$) of $\mathbf{q}$; the red curve, of equation $k=ay_r$, indicates the  transcritical bifurcation (T) involving the two-fold singularity $\mathbf{p}_t$ and the pseudo-equilibrium $\mathbf{q}$; the purple curve, numerically obtained, indicates the Homoclinic loop bifurcation (HC) of $C$; and the green curve indicates the transition of the pseudo-equilibrium $\mathbf{q}$ from a node to a focus. 

  \begin{figure}[h!]
  	\begin{center}
  		\includegraphics[scale=0.6]{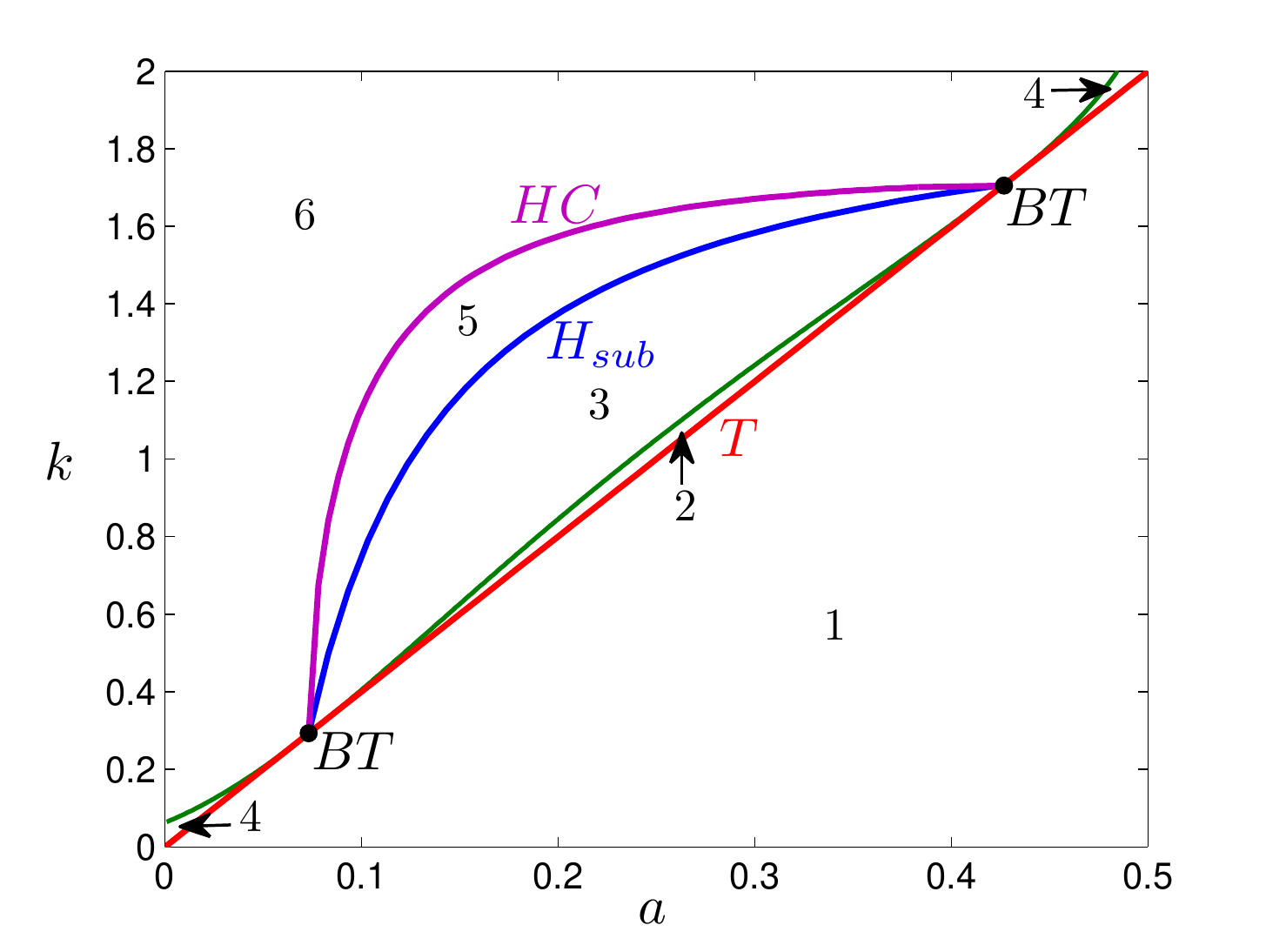}
  		\caption{Bifurcations set in $(a,k)-$plane.}\label{fig_set_bif}
  	\end{center}
  \end{figure}
  
    A point $(a,k)$ on region 1 means that $\mathbf{q}$ is a pseudo-saddle, on region 2 means that it is an unstable pseudo-node, on region 3 means that it is an unstable pseudo-focus, on region 4 means that it is a stable pseudo-node and on regions 5 and 6 means that it is a stable pseudo-focus (on region 5 there exists an unstable limit cycle). Moreover, on region 1 we get $\mathbf{q}\in\Sigma^e$ and on regions 2, 3, 4, 5 and 6 we get $\mathbf{q}\in\Sigma^s$.
 
  Based on this bifurcations analysis, we can determine an appropriate value for the control parameter $k$, from a prior knowledge of the variation range of load parameter $a$, such that the boost converter maintains the desired operating point after a load disturbance. So we should choose a value to $k$ in the region 6 (safe region) shown in Figure \ref{fig_set_bif}.

\section{Conclusion}\label{conclusao}

In this paper, we proved the existence of the bifurcations Hopf and Homoclinic Loop in the boost converter model with sliding mode control and Washout filter. The Hopf bifurcation occurs in the sliding vector field and is analogous to the standard case. The limit cycle that arises from the Hopf bifurcation is unstable and confined to the switching manifold. The Homoclinic Loop bifurcation occurs when the limit cycle disappears to touch visible-invisible two-fold point, whose dynamics in the sliding region is of the saddle type. The homoclinic loop has sliding segment which itself closes at the two-fold singularity.

The result of the bifurcation analysis was summarized in the $(a, k)$-plane of parameters. That way, is possible to choose an appropriate value for the control parameter in order to ensure operating point stability and prevent the birth of limit cycle around him, even after a change in the load parameter $a$.

\vspace{1cm}


{\small\noindent{\textbf{Acknowledgments.} The first author is supported by CAPES-Brazil. The second author is partially supported by grant 2014/02134-7, S\~{a}o Paulo Research Foundation (FAPESP) and a CAPES-Brazil grant number 88881.030454/2013-01 from the program CSF-PVE. The third author is supported by FAPEG-BRAZIL grant 2012/10 26 7000 803 and PROCAD/CAPES grant 88881.0 684 62/2014-01. The fourth author is supported by CNPq/Universal and by CAPES-BRAZIL. The second and third authors are supported by CNPq-BRAZIL grant 478230/2013-3 and all authors are supported by CNPq-BRAZIL grant 443302/2014-6. This work is partially realized at UFG as a part of project numbers 35796, 35797 and 040393.}}

\end{document}